\documentclass[a4paper,oneside,12pt]{amsart}

\usepackage{amssymb}
\usepackage{amsthm}
\usepackage[english]{babel}
\usepackage[nodate]{datetime}
\usepackage{dsfont}
\usepackage{graphicx}
\usepackage{hyperref}
\usepackage[latin1]{inputenc}
\usepackage{srcltx}
\usepackage[normalem]{ulem}
\usepackage{vruler}

\usepackage[hmargin=30mm,top=30mm,bottom=35mm]{geometry}

\reversemarginpar

\theoremstyle{plain}
\newtheorem{theorem}{Theorem}
\newtheorem{lemma}[theorem]{Lemma}

\newtheorem{proposition}[theorem]{Proposition}
\newtheorem*{proposition*}{Proposition}
\newtheorem*{theorem*}{Theorem}
\theoremstyle{definition}

\newtheorem*{definition*}{Definition}
\theoremstyle{remark}

\newtheorem*{claim*}{Claim}

\newtheorem*{remark*}{Remark}

\newcommand{\N}{\mathbb{N}}

\newcommand{\Z}{\mathbb{Z}}

\newcommand{\I}{\mathds{1}}

\newcommand{\dd}{{\mathrm d}}

\renewcommand{\ge}{\geqslant}
\renewcommand{\le}{\leqslant}
\renewcommand{\geq}{\geqslant}
\renewcommand{\leq}{\leqslant}

\newcommand{\abs}[1]{\left\lvert{#1}\right\rvert}

\begin{document}

\title{Stochastic Perturbations of Convex Billiards}
\author{Roberto Markarian, Leonardo T. Rolla, Vladas Sidoravicius, Fabio A. Tal, Maria E. Vares}

\maketitle

{
\vspace{-5mm}
\centering
\scriptsize
IMERL, Facultad de Ingenier\'{\i}a, Universidad de la Rep\'{u}blica; 
Instituto de Investigaciones Matem\'aticas Luis A. Santal\'o, Consejo Nacional de Investigaciones
Cient\i{\i}ficas y T\'ecnicas and Universidad de Buenos Aires; 
Instituto de Matem\'atica Pura e Aplicada; 
Intituto de Matem\'atica e Estat\'{\i}stica, Universidade de S\~ao Paulo; 
Instituto de Matem\'atica, Universidade Federal do Rio de Janeiro.
}

\begin{abstract}
We consider a strictly convex billiard table with $C^2$ boundary, with the dynamics subjected to random perturbations.
Each time the billiard ball hits the boundary its reflection angle has a random perturbation.
The perturbation distribution corresponds to a situation where either the scale of the surface irregularities is smaller than but comparable to the diameter of the reflected object, or the billiard ball is not perfectly rigid.
We prove that for a large class of such perturbations the resulting Markov chain is uniformly ergodic, although this is not true in general.
\vspace{2mm}

\end{abstract}

This preprint has the same numbering of sections, figures and theorems as the the
published article
``\emph{Nonlinearity 28 (2015), 4425-4434.}''

\vspace{1mm}

\section{Introduction}

Billiards with a stochastic perturbation of the outgoing angle are very natural models motivated by microscopic kinetic problems, theoretical computer science, etc., and have been receiving increased attention from deterministic and stochastic dynamics communities in the last decade.
In most of the studied cases, the outgoing angle is either uniformly distributed, or it is chosen according to the Knudsen cosine law, see for instance~\cite{evans-01,feres-zhang-12,khanin-yarmola-13}.
These types of reflection laws are physically relevant for billiards where the micro-structure and irregularities of the boundaries have a typical length-scale larger than the diameter of the billiard ball.

In the present work we focus on the case of stochastic perturbations of the classical deterministic billiard inspired by a physical situation where either the billiard ball is not perfectly rigid, or the scale of the surface irregularities is smaller than but comparable to the diameter of the reflected object, see Figure~\ref{figure1}.

Deterministic billiards on sufficiently smooth strictly convex tables are non-ergodic~\cite{lazutkin-73,katok-hasselblatt-95}.
We show that, in contrast to the deterministic situation, for a certain class of physically relevant stochastic perturbations of a reflection law, the associated Markov process is uniformly ergodic, and that any probability measure is attracted exponentially fast to a unique invariant probability measure.
We note that this is not true in general: there are examples of stochastic perturbations under which the resulting system is not ergodic.
Our result holds for billiard tables which are strictly convex, with $C^2$ boundary, including the possibility of isolated points of null curvature.

The mathematical setup is informally described as follows.
Given a prescribed family of independent random variables $\{Y_{\theta} \}_{\theta \in [0, \pi]} $, the dynamics obeys a stochastic rule.
If the outgoing angle after a deterministic collision would have been $\theta$, it is taken as $\theta + Y_{\theta}$ instead.
The family $\{Y_{\theta} \}_{\theta \in [0, \pi]}$ is chosen in such a way that typically the influence of $Y_{\theta}$ is negligible compared to $\theta$.
However, it becomes substantial when the incidence angle gets too small.
The latter property reflects an increased sensitivity to surface rugosity. 

Stochastic perturbations of classical billiard systems have been proposed and studied before, see~\cite{comets-popov-schutz-vachkovskaia-09,feres-zhang-10,cook-feres-12,chumley-cook-feres-13,yarmola-13}.
Yet, we are not aware of prior studies of perturbations similar to those considered here.

This paper is divided as follows.
In Section~\ref{sec:models} we describe the model and state the main result. Basic properties of convex billiard tables are briefly reviewed in Section~\ref{sec:deterministic}. The main  result of this paper is Proposition~\ref{pr:principal}, which is stated in Section~\ref{sec:markov}, and from which Theorem~\ref{theo-rolandgarros} follows at once by classical results. Section~\ref{sec:proof} is dedicated to the main technical proofs, culminating with the proof of Proposition~\ref{pr:principal}.  A quantitative version of this proposition, stated as Theorem~\ref{thm:epsilonsquared},
is given in Section~\ref{sec:quantitative}.

\begin{figure}[t!]
\hfill
\includegraphics[height=28mm]{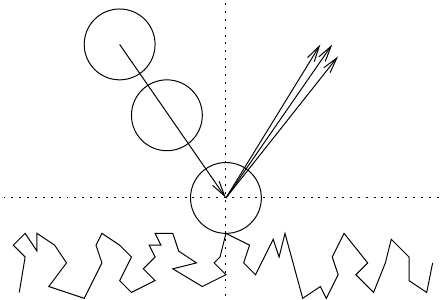}
\hfill
\includegraphics[height=28mm]{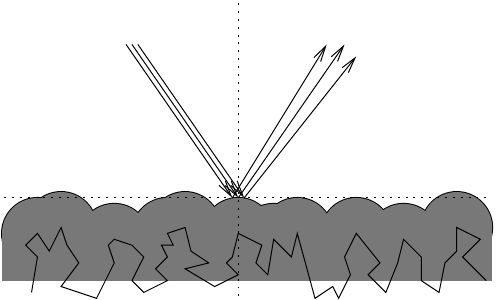}
\hfill
\includegraphics[height=28mm]{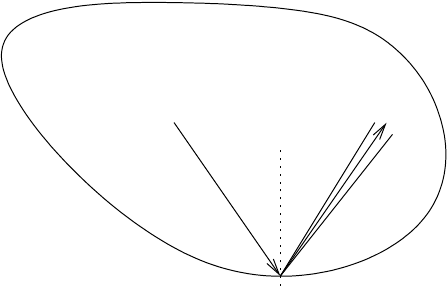}
\hfill{}
\caption{\small Perturbation on the reflexion angle motivated by microscopic roughness and ball radius.
From left to right: microscopic paradigm of round particle reflecting on a rough surface; equivalent microscopic model of point particle reflecting on a smooth surface; effective macroscopic model.}
\label{figure1}
\end{figure}

\section{Models and result}
\label{sec:models}

We begin with the description of the deterministic billiard in $D$, a connected domain in $\mathbb{R}^2$.
\emph{We assume throughout this paper that~$D$ is strictly convex with $C^2$ boundary.}
Notice that isolated points with null curvature are allowed.

The billiard in $D$ is the dynamical system describing the free motion of a point mass inside $D$ with elastic reflections at its boundary $\Gamma$.
Let $n(q)$ be the unit normal to the curve $\Gamma$ at the point $q$ pointing towards the interior of $D$.
The phase space of such a dynamical system is
\(
\{(q,v): \, q \in \Gamma, \, |v| =1, \langle v, n(q)\rangle \geq 0 \}.
\)

The image of a point $(q_0, v_0)$ by the deterministic billiard map $T$ is denoted by
\[
T(q_0, v_0) = (q_1, v_1)
\]
and defined as follows.
First, $q_1$ is the point where the oriented line through $(q_0, v_0)$ hits $\Gamma$.
Finally, $v_1$ is the velocity vector after the reflection at $q_1$, i.e. $v_1=v_0 - 2 \langle n(q_1), v_0\rangle n(q_1)$.

We take the set of coordinates $(s, \theta)$, where $s$ is the arc-length parameter along $\Gamma$ and $\theta \in [0, \pi]$ is the angle between $v$ and the oriented tangent line to the boundary at $q$.
The phase space under these coordinates is given by the cylinder
$$M=\{(s,\theta): 0 \leq s < |\Gamma |, 0 \leq \theta \leq \pi\}.$$
For $x = (s,\theta) \in M$, we write $s(x) = s, \theta(x) = \theta$, and also $q(x)$ for the corresponding point in~$\Gamma$.
The map $T$ is a diffeomorphism defined on the compact set $M$ with fixed points at $\partial M = \{(s,\theta):\theta =0 \text{ or } \pi\}$.

Moreover, $T$ is a {\em twist diffeomorphism}.
This means that the image of any vertical line ($s=$ constant) is a smooth curve with slope positive and bounded away from infinity, see~\cite[Section~9.2]{katok-hasselblatt-95}.

If $D$ is strictly convex with sufficiently smooth boundary, by KAM theory there exist invariant curves of the billiard map as close as we want to the boundary~$\partial M$, see the end of Section~\ref{sec:deterministic}.
Therefore, if the initial angle is small it remains small along the whole trajectory.
We show that this regularity can be broken using arbitrarily small random perturbations.

\medskip

We consider the system with random perturbations that act on the outgoing angle, independently of the position, by adding a random variable to~$\theta$.
Fix $0<\epsilon<\frac{\pi}{2}$.
When the incidence angle $\theta$ is away from $0$ and $\pi$, we take the outgoing angle uniformly distributed on the interval $[\theta-\epsilon,\theta+\epsilon]$.
Values of $\theta$ close to $0$ or $\pi$ need to be truncated, otherwise the ball would leave the billiard table.
Let $\theta^\epsilon : = \min \{\max (\theta, \epsilon), \pi - \epsilon\}$.
For every point $x=(s,\theta)\in M$, consider the measure $Q^{\epsilon}_{x}$ on $M$ given by
\begin{equation}
\nonumber
Q^{\epsilon}_{x}(A) =
\int_{\theta^\epsilon-\epsilon}^{\theta^\epsilon+\epsilon}
\I_{A}(s,u) \frac{1}{2\epsilon}
\dd u.
\end{equation}

In other words, the random outgoing angle is distributed uniformly on $[\theta^\epsilon-\epsilon,\theta^\epsilon+\epsilon]$.
This choice of $Q^\epsilon$ is discussed further below.

\medskip

Denote by $\mathcal{B}$ the Borel $\sigma$-field on $M$, and $\mathcal{P}$ the set probability measures on $\mathcal{B}$,
and the total variational distance on $\mathcal{P}$ denoted by $\|\mu-\nu\| = \sup_{A \in \mathcal{B}} |\mu(A)-\nu(A)|$. 

\begin{definition*}
The stochastic perturbation of the map $T$ is given by the transition kernel
\(
P_{\epsilon}(x,A) = Q^{\epsilon}_{Tx}(A),\
 x\in M,\ A\in \mathcal{B}.
\)
\end{definition*}

Observe that $P_{\epsilon}(.,A)$ is a measurable function for every $A \in \mathcal{B}$, and $P_{\epsilon}(x,.)$ is a measure on $\mathcal{B}$ for every $x\in M$.

The \emph{push-forward} operator $\mu \mapsto \mu P_\epsilon$ for $\mu\in\mathcal{P}$ is given by
$$\mu P_\epsilon(A) = \int _{M} \mu (\dd x) P_\epsilon(x,A),$$
and we say that $\mu\in\mathcal{P}$ is \emph{invariant for $P_\epsilon$} if $\mu P_\epsilon = \mu$, see Section~\ref{sec:markov}.

\begin{theorem}
\label{theo-rolandgarros}
Suppose that $D$ is strictly convex and its boundary~$\Gamma$ is $C^2$.
For each $0 < \epsilon < \frac{\pi}{2}$, there exists a unique invariant measure $\nu_\epsilon$ for $P_\epsilon$, and moreover there exists $\gamma_\epsilon>0$ such that
\(
\left\| \big. \mu P^n_{\epsilon} - \nu_\epsilon \right\| \leqslant e^{-\gamma_\epsilon n} \; \:
\mbox{for all $\mu\in\mathcal{P}$ and $n\in\N$}.
\)
\end{theorem}

The proof of Theorem~\ref{theo-rolandgarros} will follow from classical results and from two propositions. The first, Proposition~\ref{pr:existedensidade}, shows that for $n\ge2$ the transition kernel  for the iterated process $P^n_\epsilon(x, A)$ has an associated density function $p^{n}_\epsilon(x,y)$, and the second, Proposition~\ref{pr:principal}, shows that there exist an uniform coupling time, that is, for some $N\ge 2$, the density function $p^N_\epsilon(x,y)$ is uniformly bounded away from $0$. The proof of Proposition~\ref{pr:principal} is done by topological methods when working in the full generality of the hypotheses from Theorem ~\ref{theo-rolandgarros}, but with a slightly more restrictive condition we can derive a quantitative version of Proposition~\ref{pr:principal}, where the number of iterates needed to obtain a strictly positive density function is bounded by a function of $\epsilon$.

\begin{theorem}
\label{thm:epsilonsquared}
Suppose that $D$ is strictly convex and its boundary~$\Gamma$ is $C^2$, with nowhere null curvature.
Then there exists $k>0$ such that, for every $0<\epsilon <\pi/2$, if $n>2+\frac{k}{\epsilon^2}$, there exist $b=b(\epsilon) > 0$ such that the density (see Section 4) verifies 
\begin{equation}
\nonumber
p^n_\epsilon(x,y) > b, \quad \forall \, x,y \in M. 
\end{equation}
\end{theorem}

\medskip

The transition kernel $Q^\epsilon$ considered here is a small random perturbation of a close to integrable Hamiltonian system.
The invariant measure~$\nu_\epsilon$ does not satisfy the so-called cosine law exactly, but does approximately as $\epsilon$ is chosen small enough.

Our choice of this particular perturbation is motivated by physical situations where either the scale of the surface irregularities is smaller than but comparable to the diameter of the reflected object, or the billiard ball is not perfectly rigid.
Collision of a round ball with a rough surface is equivalent to the collision of a point particle with the convolution of the surface with a ball.
In the scenario depicted on Figure~\ref{figure1}, such a point particle would see the topmost part of a discrete set of circles, and the random deviation comes from the uncertainty about the slope at the point where the particle hits the circle.
Only part of the circles is exposed to the particle, and the range of possible slopes depends on the ratio between the scale of roughness and the radius of the ball.

\begin{figure}[ht!]
\includegraphics[width=90mm]{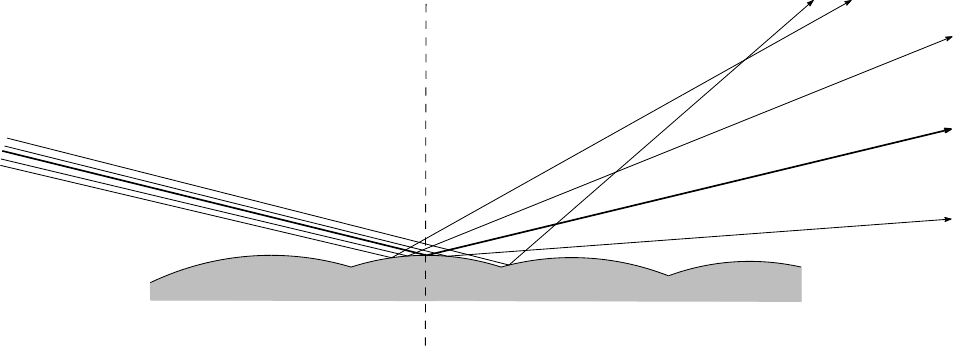}
\hfill
\includegraphics[width=55mm]{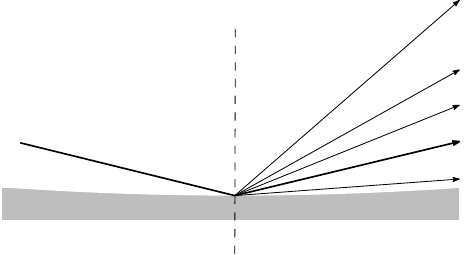}
\caption{\small Left: mesoscopic view of a round ball colliding with a rough surface (or, equivalently, point mass colliding with an irregular but microscopically smooth surface). Right: macroscopic effective model for the same collision. The thicker trajectory has equal incidence and outgoing angles.}
\label{figure2}
\end{figure}

When the incidence angle $\theta$ is far from $0$ and $\pi$, the choice of a constant density on $[\theta-\epsilon,\theta+\epsilon]$ is not particularly important for the qualitative behavior of the system, and is in fact irrelevant in our proof.
The important property is that the density is bounded from below within a certain distance from $\theta$, uniformly over $x\in M$.

When the incidence angle is very small, the randomness of the outgoing angle is no longer determined by how the particle hits a given protuberance: it is rather sensitive to which parts of the protuberances are visible to the particle. For the reflected angle to be yet smaller than the incidence angle, it would require the particle to hit the surface on the back side, which becomes less likely as the incidence angle tends to zero.
This explains the asymmetry seen in Figure~\ref{figure2}.

\medskip

Theorem~\ref{theo-rolandgarros} remains valid, with essentially the same proof, for a much broader class of distributions.
What is relevant to the proof is that the probability density of the outgoing angle is bounded from below on some interval around $\theta$ whose length is also bounded from below.

Yet, the validity of Theorem~\ref{theo-rolandgarros} is far from being general.
For instance, if one takes the outgoing angle uniformly distributed on
$[0,2\theta]$ for $\theta<\epsilon$, that is, symmetric around $\theta$, the resulting stochastic dynamics is not only non-ergodic, but it gets quickly absorbed by a random point at the boundary~$\partial M$.
We omit the proof of this fact.

\section{Basic properties of deterministic billiards in convex tables}
\label{sec:deterministic}

The map $T$ preserves the probability measure
$\nu$ defined by $d\nu = \frac{1}{2|\Gamma|} \ \sin \theta \ ds \ d\theta.$
It satisfies an involution property: if $I: M\to M$ is defined by $I(q, \theta) =
(q, \pi - \theta)$, then $T^{-k} \circ I = I \circ T^k, k \in \Z.$

If $k \geq 2, T$ is a $C^{k-1}$ diffeomorphism in the interior of $M$.
For every $x \in M$, the matrix of the differential $ D_{x}T $ reads as
\begin{equation}
\nonumber
D_{x}T =
\begin{pmatrix}
\frac {\kappa(x) t(x) - \sin \theta(x)}{\sin \theta(Tx)} &
\frac {t(x)}{\sin \theta(Tx)} \\
\frac{\kappa(Tx) \kappa(x) t(x) - \kappa(Tx) \sin \theta(x)}
{\sin \theta(Tx)} - \kappa(x) & \frac{\kappa(Tx) t(x)}{\sin \theta(Tx)} - 1
\end{pmatrix},
\end{equation}
where $ \kappa(x) $ is the curvature of $ \Gamma $ at $ q(x) $, and $t(x)$ is the
distance between $q(x)$ and $q(Tx).$ Both values are continuous for
$x \in \textrm{int} M,$ the interior of $M$ (see~\cite{chernov-markarian-06} for a proof, noticing that here we use a different parametrization of angles).

\begin{lemma}\label{imageb}
If $T(x) = (s_1(x), \theta_1(x))$,  then $\left(\frac{\partial s_1}{\partial \theta}\right)^{-1}$ can be continuously extended to the boundary of $M$. In particular, $\frac{\partial s_1}{\partial \theta}$ is bounded away from zero.
\end{lemma}

\begin{proof}
Using the expression of $D_xT$ we obtain that $\frac{\partial s_1}{\partial \theta}= \frac {t(x)}{\sin \theta_1(x)}$. 
If $x_n$ is a sequence converging to a boundary point $\overline{x}$, then both $t(x_n)$ and $\sin \theta_1(x_n)$ converge to $0$, but in
this case $\frac{\kappa(Tx_n) t(x_n)}{2\sin \theta(Tx_n)}$ tends to 1 and the $\left(\frac{\partial s_1}{\partial \theta}\right)^{-1}$ goes to $\kappa(T\overline{x})/2$.
\end{proof}

We remark that $T$ being a twist map holds on more general tables, in particular if $D$ is convex with $C^1$ boundary.
In this case the map $T$ is an homeomorphism in $\textrm{int}M$ that can be extended defining $Tx = x$ for every $x \in \partial M$.
For this extension $x \mapsto q(T(x))$ is not continuous in $x$ if $q(x)$ is in the interior of a segment of the boundary, but $x \mapsto \theta(T(x))$ is nonetheless continuous.

For any $0<a<\frac{\pi}{2}$, we define the cylinder $M_a=[0, |\Gamma |)\times [a, \pi-a]$.

\begin{lemma}
\label{c}
Suppose that $D$ is convex with $C^1$ boundary.
Given $\epsilon > 0$, there
exist $0 <c_2 < c_1 < \epsilon$ satisfying the following conditions:
$M_{\epsilon}\subset T(M_{c_1})$, $M_{c_1}\subset T^2(M_{c_2})$ and
$T^2(M_{c_1})\subset M_{c_2}$.
\end{lemma}
\begin{proof}
Although $T$ may not be continuous in $M$, $\theta (Tx)$ is still continuous. This implies that for sufficiently small $a > 0$ there exists  $\delta_1(a), \delta_2(a)>0$ such that $M_{\delta_1(a)}\subset T(M_a)\subset M_{\delta_2(a)}$, and that $\delta_1(a), \delta_2(a)$ converge to $0$ as $a \to 0$.
\end{proof}

For the deterministic billiard on a sufficiently smooth table, Lazutkin~\cite{lazutkin-73} proved the following regularity result, see also~\cite{douady-82}.
If $D$ is convex with smooth boundary and curvature bounded from below then there exists a subset $M'$ of the phase space $M$ that has positive measure and is foliated by invariant curves; the set $M'$ accumulates on the horizontal boundaries of $M$, the map $T$ restricted to each such curve is topologically equivalent to an irrational rotation; close to the boundary ($\theta=0$ or $\pi$ in the phase space) there is a set of positive measure with regular behavior.
In fact in the circle or the ellipse the whole phase space is foliated by invariant curves.
Theorem~\ref{theo-rolandgarros} shows that this regularity can be broken by an arbitrarily small stochastic perturbation.

There are billiards on convex regions with no invariant curves near the boundary.
These billiards have trajectories with an infinite number of bounces in finite time
as they approach a point of the boundary.
They can be constructed either violating the condition on the curvature or the differentiability of the boundary.
In~\cite{halpern-77}, Halpern constructed a curve that has nowhere vanishing curvature but unbounded third derivative, and proved that there are trajectories bearing this pathological behavior.
Mather~\cite{mather-82} constructed a convex billiard with $C^2$ boundary violating the condition of non-null curvature, and which has trajectories coming arbitrarily close to being positively tangent to the boundary and then arbitrarily close to being negatively tangent to the boundary.

\section{Markov chains and their densities}
\label{sec:markov}

Recall that the stochastic perturbation of the map $T$ is given by the transition kernel
\[
P_{\epsilon}(x,A) = Q^{\epsilon}_{Tx}(A),
\quad x\in M,\ A\in \mathcal{B}.
\]

Let $ P^n_\epsilon$ denote the $n$-th power of the kernel
$$
P^{n+1}_\epsilon(x,A) =\int_M P_\epsilon (x,\dd y) P^n_\epsilon(y,A).
$$ 
$P^n_{\epsilon}(x,.)$ is a probability measure on $\mathcal{B}$ for every $x\in M$.
Moreover, as operators on $\mathcal{P}$ they satisfy $\mu P^n_\epsilon = (\mu P^{n-1}_\epsilon) P_\epsilon$, and defining $P_\epsilon^0(x,A) = \I_{A}(x)$, the set $(P^n_\epsilon)_{n\in\N_0}$ forms a semi-group.

\begin{proposition}\label{pr:existedensidade}
For the stochastic billiard map there exist density functions $p^n_{\epsilon}(x,y)$ such that,
for every $x\in M$, $n\ge 2$, and $A \in \mathcal{B}$,
$$
P^n_{\epsilon}(x,A) = \int_A p^n_{\epsilon}(x,y) \dd y
.
$$
\end{proposition}

\begin{proof}
If $A=[\tilde s, \hat s]\times [\tilde \theta, \hat \theta], x= (s, \theta),
Tx = (s_1, \theta_1), z = (s_1, \theta'), Tz =(s'_1, \theta_1')$ then
$$P^2_\epsilon(x, A) = \frac 1 {4\epsilon ^2}\int_{[\theta^{\epsilon} _1 -\epsilon, \theta^{\epsilon}_1 + \epsilon]}d \theta '\int_{[-\epsilon, \epsilon]}
I_A (s'_1, \theta'^{\epsilon}_1 + u) \dd u.
$$ 
Changing variables $d\theta' =  \frac{\partial \theta'} {\partial s'_1} \; ds'_1$,
and using Lemma~\ref{imageb},
we obtain the desired density.
The general case is analogous.
\end{proof}

A Markov chain is said to satisfy \emph{D\"oblin's condition} if there exists a probability measure $\lambda$, $n>0$ and $\delta_1<1,\,\delta_2>0$ such that, whenever $\lambda(A)>\delta_1$, then $P^n_{\epsilon}(x, A)>\delta_2$, for all $x\in M$.

Theorem~\ref{theo-rolandgarros} is a consequence of the following result.

\begin{proposition}
\label{pr:principal}
Suppose that $D$ is strictly convex and its boundary~$\Gamma$ is $C^2$.
Then for every $0<\epsilon <\pi/2$, there exist $b > 0$ and $N > 0$ such that 
\begin{equation}
\nonumber
p^N_\epsilon(x,y) > b, \quad \forall \, x,y \in M. 
\end{equation}
\end{proposition}
We postpone its proof to the next section.

\begin{proof}
[Proof of Theorem~\ref{theo-rolandgarros}]
Proposition~\ref{pr:principal} implies that the chain is aperiodic, $\psi$-irreducible and satisfies D\"oblin's condition; then it is uniformly ergodic, see~\cite[Theorem~16.2.3]{meyn-tweedie-09}.
The result then follows from~\cite[Theorem~16.0.2]{meyn-tweedie-09}.
\end{proof}

\section{Proof of D\"oblin's condition}
\label{sec:proof}

Before proving Proposition~\ref{pr:principal} and Theorem~\ref{theo-rolandgarros}, we need a few additional technical steps, summarized in the next two propositions.

\begin{definition*}
We say that a sequence $\xi = (\xi_k)_{k\in{[0,\dots,l]}},\, \xi_k=(s_k, \theta_k) \in M$ is \emph{an $\epsilon$-angular perturbed orbit of length $l$} if, for all $0\le k<l$, $s(T(\xi_k))=s_{k+1}$ and $\abs{\theta(T(\xi_k))-\theta_{k+1}}<\epsilon$. 
\end{definition*} 

Let $\mathcal{O}_{\epsilon,l}$ denote the set of $\epsilon$-angular perturbed orbits of length $l$.
For $n>0$ define
\begin{equation}
\nonumber
\widehat{T}_\epsilon^n(x):=\{y \, : \; \exists \, \xi \in \mathcal{O}_{\epsilon,n}, 
\text{ such that } \xi_0=x, \xi_n=y\}.
\end{equation}
Starting at a point $x$, $\widehat{T}_\epsilon^n(x)$ is the set of points that may be reached in $n$ steps by following the deterministic billiard but allowing for perturbations smaller than
$\epsilon$ in the reflection angle.

\begin{proposition}
\label{pr:aux1}
If $0<l<n, n\ge 3,\, y\in \widehat{T}_\epsilon^n(x)$ and there exists $z$ in the interior of 
$M$ such that $z\in \widehat{T}_\epsilon^{l}(x)$ and $y\in \widehat{T}_\epsilon^{n-l}(z)$, then 
$p^n_{\epsilon}$ is continuous at $(x,y)$ and $p^n_{\epsilon}(x,y)>0$.
\end{proposition}
\begin{proof}
From the above definition we have that
$\widehat{T}_\epsilon^n(x)= \bigcup_{z\in \widehat{T}_\epsilon^{n-1}(x)}\widehat{T}_\epsilon(z)$, and if $y\in \widehat{T}_\epsilon^2(x)$ 
and $y$ does not belong to the boundary of $M$, then $y$ belongs to the interior of the support of $p^2_{\epsilon}(x, .)$ and $p^2_{\epsilon}$ is continuous at $(x,y)$.
\end{proof}

\begin{proposition}
\label{previous}
Suppose that $D$ is convex with $C^1$ boundary given by the union a finite number of $C^2$ arcs and line segments.
For every $\epsilon > 0$, there exists $N \in\mathbb{N}$ such that, for all $x,y\in M$, there exists $\xi \in \mathcal{O}_{\epsilon,N}$, such that $\xi_0=x$ and $\xi_N=y$.
\end{proposition}
\begin{proof}
We split the proof in two steps.
First we show that it is possible to move between points in a given small neighborhood.
Finally we use this fact to cover the whole phase space.

\emph{Step 1.}
By definition, if $\theta_1 = \theta (Tx)\in [\epsilon, \pi - \epsilon]$,  or equivalently $Tx \in M_\epsilon$, then $\widehat{T}_\epsilon x = \{s_1\}\times(\theta_1 - \epsilon, \theta_1 + \epsilon)$.
In any case 
$\widehat{T}_\epsilon x = \{s_1\} \times (\theta_1 - \epsilon, \theta_1 + \epsilon)\cap M$, and $\widehat{T}_\epsilon^2 x$ is a distorted rectangle.
If $x$ does not belong to the boundary of $M$, then $T^2x$ lies 
in the interior of $\widehat{T}_\epsilon^2 x$. Now, take $c_2 (\epsilon) > 0$, fixed by the 
modulus of continuity of $T$ as in Lemma~\ref{c}. If $\delta < c_2$ is sufficiently small, then
for all $x$ in $M_{c_2}$, both $T^2(B_{2\delta}(x))$ and $B_{2\delta}(T^2x)$ are contained 
in $\widehat{T}_\epsilon^2 x$.

Consider a set $U\subset M_{c_1}$ with measure $\nu(U)>0$ and with diameter smaller
than $\delta$, and let $x_1$ be a point in $U$. As a consequence of the Poincar\'e
Recurrence Theorem and Birkhoff-Khinchin Ergodic Theorem, there exists a point $z$
in $U$ and $n_{_U}\le (\nu(U))^{-1}$ with $T^{n_{_{U}}}(z)$ in $U$.

By choice of $\delta$, we have that $z_2=T^2(z)\in \widehat{T}_\epsilon^2(x_1)$.
From this we have that $T^{n_U-4}(z_2)=T^{n_U-2}(z)$ belongs to $\widehat{T}_\epsilon^{n_U-2}(x_1)$.
Note that, by the choice of $c_2$, since $T^{n_U}(z)$ belongs to $M_{c_1}$, then
$T^{n_U-2}z$
belongs to $M_{c_2}$ and so, again by the choice of $\delta$, the ball
of radius $2\delta$ and center $T^2T^{n_U-2}(z)$ is contained $\widehat{T}_\epsilon^2(T^{n_U-2}z)$
and so $U\subset \widehat{T}_\epsilon^2(T^{n_U-2}z)\subset \widehat{T}_\epsilon^{n_U}x_1$. Therefore our dynamics moves
any point of $U$ to any other point in $U$ by the step $n_U \le (\nu(U))^{-1}$.

\emph{Step 2.}
We partition the cylinder $M_{c_1}$ into $k$ rectangles $R_1,\dots, R_k$ based 
on a rectangular grid of size less then $\delta/2$, and consider the collection $Q_1,\dots,Q_l$ of
rectangles of diameter less than $\delta$, made of two adjacent rectangles $R_i, R_j$.

Let $N_0$ be such that $N_0^{-1}$ is smaller that the minimum of
$\nu({Q_j}),\,1\le j\le l$. Then $N_0$ only depends on $\epsilon$, and for each $Q_j$,
there exists $n_{Q_j}<N_0$ such that, for any two points $x, y$ in $Q_j$, $y$ belongs
to $\widehat{T}_\epsilon^{n_{Q_j}}(x)$. Let $N_1$ be the least common multiple of $\{1,\dots, N_0-1\}$. 
Then, repeatedly applying the same reasoning in each rectangle, any two points in
the same rectangle can be joined by a random trajectory at step $N_1$. More precisely,
$\widehat{T}_\epsilon^{N_1}x$ contains $Q_j$ for each $x \in Q_j.$

Consider two points $x_0, y$ in $M_{c_1}$. There exists a sequence of adjacent
rectangles $R_0, R_1, \dots, R_m,\, m< k$, such that $x_0 \in R_0$ and $y \in R_m$. 
Choose $x_i \in R_i,\, 1\le i\le m-1$ and let $x_m=y$. By construction, for any $0\le i\le m-1$
there exists $j_i$ such that both $x_i$ and $x_{i+1}$ belong to $Q_{j_i}$. Thus
$x_{i+1} \in \widehat{T}_\epsilon^{N_1}(x_i)$. By induction, $x_m \in \widehat{T}_\epsilon^{m N_1}(x_0)$.
On the other hand, as a consequence of the recurrence of $R_m$ by $\widehat{T}_\epsilon$,
$x_m \in \widehat{T}_\epsilon^{N_1}(x_m)$ and so we have that $x_m \in \widehat{T}_\epsilon^{n N_1}(x_0)$
for any $n>m$.
Since $x_0$ and $y$ were arbitrary and $m<k$, for any $x$ in $M_{c_1}$, $\widehat{T}_\epsilon^{kN_1}(x)$ contains $M_{c_1}$.

For any $x$ in $M, \, \widehat{T}_\epsilon x$ intersects $M_{\epsilon}\subset M_{c_1}$, and so $\widehat{T}_\epsilon^{kN_1+1}x$ contains $M_{c_1}$. Now observe that if $\theta(Tx)$ is smaller than $\epsilon$ or greater than $\pi - \epsilon$, then $\widehat{T}_\epsilon x$ contains the segment $(s(Tx) \times [0, \epsilon])$ or the segment $(s(Tx) \times [\pi -\epsilon, \pi])$.
Then, since $M_{\epsilon}\subset T(M_{c_1}),\, \widehat{T}_\epsilon(M_{c_1})= M$ we have that $\widehat{T}_\epsilon^{kN_1+2}x=M$ for any $x \in M$.
\end{proof}

\begin{proof}
[Proof of Proposition~\ref{pr:principal}]
By Proposition~\ref{previous}, 
there exists $N\ge 2$ such that, for all $x,y \in M,\, y\in \widehat{T}_\epsilon^N(x)$. In particular, 
by Proposition~\ref{pr:aux1}, $p^N_{\epsilon}$ is continuous and strictly positive in $(x,y)$. 
The result follows by compactness.
\end{proof}

\section{A quantitative estimate in D\"oblin's condition}\label{sec:quantitative}

In this section we present, in a more restrictive setting, an upper bound depending on $\varepsilon$ for the minimal value $N$ in Proposition \ref{pr:principal}. In what follows, we assume that $D$ is convex with $C^2$ boundary, and that the curvature of $\Gamma$ is nowhere null. In this setting, $T$ is a diffeomorphism of the  closed annulus $M$. Let $L=\max\{\Vert DT\vert, \Vert DT^{-1}\Vert, 2^{\frac{1}{5}}\}$.  Since the boundaries of $M$ are invariant by $T$, it follows that for any $a\in [0, \frac{\pi}{2L}]$, $M_{La}\subset M_{a}\subset M_{\frac{a}{L}}$. In particular, we will set $c_1= \frac{\varepsilon}{L}$ and $c_2=\frac{\varepsilon}{L^3}$, and it is immediate that  both  satisfy the conditions of  Lemma \ref{c}. 

Since $\frac{\partial s_1}{\partial \theta}$ is bounded away from $0$ and infinity, there exists some constant $d_1\ge1$ such that the image of any vertical line ($s = $ constant)  has slope in the interval $[\frac{1}{d_1}, d_1]$.

\begin{lemma}\label{lm:deltaestimate}
If $\delta \le \frac{c_2}{d_1L^2}\le\frac{\varepsilon}{d_1L^4}$, then
for all $x$ in $M_{c_2}$, both $T^2(B_{2\delta}(x))$ and $B_{2\delta}(T^2x)$ are contained 
in $\widehat{T}_\epsilon^2 x$.
\end{lemma}

\begin{proof}
If $x\in M_{c_2}$, then $Tx\in M_{c_2/L}$ which implies that the segment $\{s_1\}\times (\theta_1- c_2/L, \theta_1+ c_2/L)$ is contained in $\hat T_{\epsilon}(x)$. The image of the segment  is a curve with slope at most $d_1$ and length at least $(c_2/L)/L$ passing through $T^2x=(s_2, \theta_2)$. Therefore, for all $\theta$ such that $\vert\theta-\theta_2\vert<\frac{c_2}{d_1L^2}$, there exists some $s(\theta)$ with $\vert s(\theta)-s_2\vert \le \frac{c_2}{L^2}=\frac{\epsilon}{L^5}\le\frac{\epsilon}{2}$ such that $(s(\theta),\theta)$ lies in the image by $T$ of $\hat T_{\epsilon}x$, which implies that the rectangle $(s_2-\delta, s_2+\delta)\times (\theta_2- \epsilon_2, \theta_2+\epsilon_2) \cap M$ is contained in $\hat T^2_{\epsilon}x$.     \end{proof}

\begin{lemma}\label{lm:volumebola}
There exists some constant $d_2$ such that, for every $x\in M_{c_2}$ and any $r\le c_2$, $\nu(B_r(x))\cap M_{c_2}\ge d_2 r^2$
\end{lemma}

\begin{proof}
This follows from the fact that $\nu$ is absolutely continuous with respect to Lebesgue and its density is uniformly bounded away from $0$ on $M_{c_2}$. 
\end{proof}

\begin{proposition}\label{pr:volumecresce}
If $x\in M_{c_1}$ and $n\ge 1$ then either $M_{c_2}\subset \hat T_{\epsilon}^{2n}(x)$ or $\nu(\hat T_{\epsilon}^{2n}(x))\ge \frac{n d_2 \delta^2}{4}$. 
\end{proposition}

\begin{proof}
The proof is by induction. For $n=1$, since $x\in M_{c_1}$,  we have that $T^2x\in M_{c_2}$, and  by Lemma \ref{lm:deltaestimate} $B_{\delta}(T^2x) \subset \hat T^2_{\epsilon}x$, which implies, by Lemma \ref{lm:volumebola}, that $\nu(\hat T^2_{\epsilon}x)>d_2 \delta^2$.

Now assume the result true for $n=k-1$. If $M_{c_2}\subset \hat T_{\epsilon}^{2(k-1)}(x)$, then $\hat T_\epsilon(\hat T_{\epsilon}^{2(k-1)}x)=M$ and so for all $n>2(k-1),\, \hat T_{\epsilon}^{n}x=M$. If not, then $\nu(\hat T_{\epsilon}^{2(k-1)}(x))\ge \frac{(k-1) d_2 \delta^2}{4}$ and, as $T$ preserves $\nu$, then $\nu(T^2(\hat T_{\epsilon}^{2(k-1)}(x)))=\nu(\hat T_{\epsilon}^{2(k-1)}(x))$. 

Note that, for all $z\in T^2(\hat T_{\epsilon}^{2(k-1)}(x)) \cap M_{c_2}$, $B_\delta(z)\subset \hat T_{\epsilon}^2z\subset \hat T^{2k}_{\epsilon}x$.  Also, since for all $y\in M, \hat T_{\varepsilon} y$ intersects $M_{\epsilon}$, it  follows that for all $y$ and all positive $j$, $\hat T^{j}_{\epsilon}y$ intersects $M_{\epsilon}$, and therefore $T^2(\hat T_{\epsilon}^{2(k-1)}(x)) \cap M_{c_2}$ is not empty.  Let 
$$K= T^2(\hat T_{\epsilon}^{2(k-1)}(x)) \cup \left(\bigcup_{ z\in T^2(\hat T_{\epsilon}^{2(k-1)}(x)) \cap M_{c_2}} B_{\delta}(z) \right)$$
and note that $K\subset \hat T_{\epsilon}^{2k}(x).$

If $M_{c_2}\subset K$, we are done. Otherwise, there exists some $y_1\in M_{c_2}\setminus K$, which implies that $B_{\delta}(y_1)\cap M_{c_2}$ is disjoint from $T^2(\hat T_{\epsilon}^{2(k-1)}(x)) \cap M_{c_2}$.  Since the later set is not empty, one can find a point $y_2$ in $M_{c_2}$ such that the distance of $y_2$ to $T^2(\hat T_{\epsilon}^{2(k-1)}(x)) \cap M_{c_2}$ is precisely $\delta/2$. But this implies that $B_{\delta/2}(y_2)\cap M_2$ is disjoint from $T^2(\hat T_{\epsilon}^{2(k-1)}(x))$, but a subset of $K$. Therefore
$$\nu(K)\ge \nu(T^2(\hat T_{\epsilon}^{2(k-1)}(x)))+\nu(B_{\delta/2}(y_2)\cap M_2)\ge \frac{(k-1) d_2 \delta^2}{4}+ \frac{d_2 \delta^2}{4},$$ where the last inequality follows from Lemma \ref{lm:volumebola}.
\end{proof}

\begin{proof}
[Proof of Theorem~\ref{thm:epsilonsquared}]
It is straightforward that, if $x \in M$, then there exists some $z\in \hat T_{\epsilon}x\cap M_{\epsilon}$.  By Proposition \ref{pr:volumecresce}, since $\nu(M)=1$, if $\frac{n_1 d_2 \delta^2}{4}= \frac{d_2 d_1^2 \epsilon^2}{4L^8}>1$, then $\hat T_{\epsilon}^{2 n_1} z$ contains $M_{c_2}$   and therefore $\hat T_{\epsilon}^{2 n_1+1} z$ contains $M$ so $\hat T_{\epsilon}^{2 n_1+2}x $ contains $M$.  The result follows with $k= \frac{4 L^8}{d_1 d_2^2}$.
 \end{proof}

\section*{Acknowledgements}

R.M. would like to thank Ya.\ G. Sinai for suggesting this problem. R.M. also would like to thank S\^{o}nia Pinto de 
Carvalho (UFMG, Belo Horizonte), Gianluigi Del Magno (UTP, Lisbon), Grupo de Investigaci\'{o}n ``Sistemas Din\'{a}micos" (CSIC, UdelaR, Uruguay), and IMPA. F.T. was partially supported by CNPq grant 304474/2011-8 and FAPESP grant 2011/16265-8. M.E.V. was partially financed by CNPq grant 304217/2011-5 and FAPERJ
grant E-26/102.338/2013.

\bibliographystyle{bib/rollaalphasiam}
\bibliography{bib/leo}

\begin{thebibliography}{CPSV09}

\bibitem[CCF13]{chumley-cook-feres-13}
{\sc T.~Chumley, S.~Cook, and R.~Feres}, {\em From billiards to
  thermodynamics}, Comput. Math. Appl., 65 (2013), pp.~1596--1613.

\bibitem[CF12]{cook-feres-12}
{\sc S.~Cook and R.~Feres}, {\em Random billiards with wall temperature and
  associated {M}arkov chains}, Nonlinearity, 25 (2012), pp.~2503--2541.

\bibitem[CM06]{chernov-markarian-06}
{\sc N.~Chernov and R.~Markarian}, {\em Chaotic billiards}, vol.~127 of
  Mathematical Surveys and Monographs, American Mathematical Society,
  Providence, RI, 2006.

\bibitem[CPSV09]{comets-popov-schutz-vachkovskaia-09}
{\sc F.~Comets, S.~Popov, G.~M. Sch{\"u}tz, and M.~Vachkovskaia}, {\em
  Billiards in a general domain with random reflections}, Arch. Ration. Mech.
  Anal., 191 (2009), pp.~497--537.

\bibitem[Dou82]{douady-82}
{\sc R.~Douady}, {\em Applications du th{\'e}or{\`e}me des tores invariants},
  master's thesis, Univ. Paris VII, 1982.
\newblock Th{\`e}se de 3{\`e}me Cycle.

\bibitem[Eva01]{evans-01}
{\sc S.~N. Evans}, {\em Stochastic billiards on general tables}, Ann. Appl.
  Probab., 11 (2001), pp.~419--437.

\bibitem[FZ10]{feres-zhang-10}
{\sc R.~Feres and H.-K. Zhang}, {\em The spectrum of the billiard laplacian of
  a family of random billiards}, J. Stat. Phys., 141 (2010), pp.~1039--1054.

\bibitem[FZ12]{feres-zhang-12}
\leavevmode\vrule height 2pt depth -1.6pt width 23pt, {\em Spectral gap for a
  class of random billiards}, Comm. Math. Phys., 313 (2012), pp.~479--515.

\bibitem[Hal77]{halpern-77}
{\sc B.~Halpern}, {\em Strange billiard tables}, Trans. Amer. Math. Soc., 232
  (1977), pp.~297--305.

\bibitem[KH95]{katok-hasselblatt-95}
{\sc A.~Katok and B.~Hasselblatt}, {\em Introduction to the Modern Theory of
  Dynamical Systems}, Cambridge University Press, New York, 1995.

\bibitem[KY13]{khanin-yarmola-13}
{\sc K.~Khanin and T.~Yarmola}, {\em Ergodic properties of random billiards
  driven by thermostats}, Comm. Math. Phys., 320 (2013), pp.~121--147.

\bibitem[Laz73]{lazutkin-73}
{\sc V.~F. Lazutkin}, {\em Existence of caustics for the billiard problem in a
  convex domain}, Izv. Akad. Nauk SSSR Ser. Mat., 37 (1973), pp.~186--216.

\bibitem[Mat82]{mather-82}
{\sc J.~N. Mather}, {\em Glancing billiards}, Ergodic Theory Dynam. Systems, 2
  (1982), pp.~397--403.

\bibitem[MT09]{meyn-tweedie-09}
{\sc S.~Meyn and R.~L. Tweedie}, {\em Markov chains and stochastic stability},
  Cambridge University Press, Cambridge, 2~ed., 2009.
\newblock With a prologue by Peter W. Glynn.

\bibitem[Yar13]{yarmola-13}
{\sc T.~Yarmola}, {\em Sub-exponential mixing of random billiards driven by
  thermostats}, Nonlinearity, 26 (2013), pp.~1825--1837.

\end{thebibliography}

\end{document}